\documentclass[11pt,a4paper]{article}

\usepackage[T1]{fontenc}
\usepackage[utf8]{inputenc}
\usepackage{lmodern}
\usepackage[margin=1in]{geometry}
\usepackage{amsmath,amssymb,amsthm,mathtools}
\usepackage{microtype}
\usepackage[hidelinks]{hyperref}
\hypersetup{pdftitle={Tameness of transitive Boolean functions},
  pdfsubject={Boolean functions, resampling dynamics and tameness}}

\newtheorem{theorem}{Theorem}[section]
\newtheorem{proposition}[theorem]{Proposition}
\newtheorem{lemma}[theorem]{Lemma}
\newtheorem{corollary}[theorem]{Corollary}
\theoremstyle{definition}
\newtheorem{definition}[theorem]{Definition}
\theoremstyle{remark}
\newtheorem{remark}[theorem]{Remark}

\newcommand{\E}{\mathbb E}
\newcommand{\PP}{\mathbb P}
\newcommand{\R}{\mathbb R}
\newcommand{\N}{\mathbb N}
\newcommand{\cN}{\mathcal N}
\newcommand{\cE}{\mathcal E}
\newcommand{\one}{\mathbf 1}
\DeclareMathOperator{\Var}{Var}
\newcommand{\ip}[2]{\langle #1,#2\rangle}
\newcommand{\norm}[1]{\lVert #1\rVert}

\title{Tameness of transitive Boolean functions}
\author{Xinyu Long \\ School of Mathematics \\ Shandong University, Jinan, Shandong, China \\ \texttt{longxinyu@sdu.edu.cn}}

\date{}

\begin{document}
\maketitle

\begin{abstract}
Let \(f_n:\{0,1\}^n\to\{0,1\}\) be Boolean functions, and let \(C_n\) count their changes in a unit time interval under stationary dynamics in which each coordinate is independently resampled at rate one from the Bernoulli\((p_n)\) distribution. We prove that, if \(f_n\) is transitive and \(n\min\{p_n,1-p_n\}^{r}\to\infty\) for every \(r>0\), then tightness of \((C_n)\) implies \(\Var(f_n)\to0\). This settles Forsström's conjecture in the \(p_n\le 1/2\) regime and extends it symmetrically to arbitrary biases.  The proof uses a mixed jump sum whose second moment is bounded by the generator of a low degree Fourier projection.
\end{abstract}


\section{Introduction}
\label{sec:introduction}
The behaviour of Boolean functions under random perturbations has been
extensively studied through noise sensitivity and noise stability, which
admit a Fourier-analytic
treatment~\cite{BenjaminiKalaiSchramm1999,SchrammSteif2010,GarbanPeteSchramm2010};
see~\cite{GarbanSteif2014,ODonnell2014} for general accounts. A natural
continuous time counterpart is obtained by independently resampling the
coordinates, a dynamics that also underlies dynamical
percolation~\cite{HaggstromPeresSteif1997}. This leads to the more
pathwise question of how often the Boolean output changes, and to the
notions of tameness and volatility studied by Jonasson and
Steif~\cite{JonassonSteif2016}. Tameness implies noise
stability~\cite[Proposition~1.13]{JonassonSteif2016}, but the converse
fails: majority at bias $1/2$ is noise stable and neither tame nor
volatile~\cite[Proposition~1.15 and the remark following its
proof]{JonassonSteif2016}, while Galicza~\cite{Galicza2020} constructed
increasing transitive sequences that are both noise stable and volatile.
Here volatility means that the number of output changes in a unit time
interval tends to infinity in probability. 

In this paper, we study when transitivity nevertheless rules out tame behaviour. Let $\cN\subseteq\N$ be infinite. For each
$n\in\cN$, let $X^{(n)}=(X_t^{(n)})_{t\ge0}$ be the stationary process on
$\{0,1\}^n$ in which each coordinate is independently resampled at rate one
from the Bernoulli$(p_n)$ distribution, where $p_n\in(0,1)$. Write
$\pi_n=\operatorname{Bernoulli}(p_n)^{\otimes n}$. Given
$f_n:\{0,1\}^n\to\{0,1\}$, set $Y_t^{(n)}=f_n(X_t^{(n)})$ and define
\begin{equation}\label{eq:count}
  C_n=\sum_{0<s<1}\one_{\{Y_s^{(n)}\ne Y_{s-}^{(n)}\}}.
\end{equation}
All limits are taken along $\cN$, and jump sums run over configuration jumps.

\begin{definition}[Jonasson and Steif~\cite{JonassonSteif2016}]\label{def:basic}
A sequence $(f_n)$ is
\emph{tame} if $(C_n)$ is tight. It is \emph{non-degenerate} if
$\liminf_{n\to\infty}\Var_{\pi_n}(f_n)>0$. A Boolean function $f_n$ is
\emph{transitive} if, for every $i,j\in[n]$, there is a permutation $\sigma$
of $[n]$ such that $\sigma(i)=j$ and $f_n(\sigma x)=f_n(x)$ for every $x$,
where $\sigma x=(x_{\sigma(k)})_{k\in[n]}$.
\end{definition}

Jonasson and Steif~\cite[Conjecture~1.21]{JonassonSteif2016}
conjectured that every non-degenerate transitive sequence with
$\mathbb E C_n\to\infty$ is not tame. Forsström~\cite{Forsstrom2020}
constructed counterexamples to this conjecture at polynomially small
biases, and conjectured instead that every non-degenerate transitive
sequence is not tame whenever $np_n^r\to\infty$ for every
$r>0$. We confirm this conjecture for $p_n\le1/2$, the following symmetric
formulation covers arbitrary biases.

\begin{theorem}\label{thm:main}
Let $(f_n)_{n\in\cN}$ be a sequence of transitive Boolean functions, and let
$b_n=\min\{p_n,1-p_n\}$. Suppose that, for every fixed $r>0$,
\begin{equation}\label{eq:bias-condition}
  nb_n^r\longrightarrow\infty.
\end{equation}
If $(f_n)$ is tame, then $\Var_{\pi_n}(f_n)\longrightarrow0.$ In particular, every non-degenerate transitive sequence satisfying
\eqref{eq:bias-condition} is not tame.
\end{theorem}

Condition~\eqref{eq:bias-condition} is equivalent to
$\log(1/b_n)=o(\log n)$. It includes biases bounded away from zero and one,
as well as $p_n=(\log n)^{-a}$ for any fixed $a>0$. 

\begin{remark}\label{rem:endpoints}
The use of $b_n$ reflects the symmetry $x\mapsto1-x$. A condition on $p_n$
alone is insufficient near one: for $n\ge2$, take $p_n=1-1/n$ and
$f_n(x)=\prod_{i=1}^n x_i$. Then $\E_{\pi_n}f_n=p_n^n\longrightarrow e^{-1}.$ Thus
\[
  \E C_n=2n(1-p_n)p_n^n\le2.
\]

Hence the sequence is tame, while
\[
\operatorname{Var}_{\pi_n}(f_n)
=p_n^n(1-p_n^n)\to e^{-1}(1-e^{-1})>0,
\]
so it is also non-degenerate.
\end{remark}

For $X\sim\pi_n$, let $R_iX$ be obtained by independently resampling
its $i$-th coordinate, and define the resampling influences by
\begin{equation}\label{eq:influence-definition}
  I_n(f_n)=\sum_{i=1}^n I_{n,i}(f_n),
\end{equation}
where $I_{n,i}(f_n)
  =\PP\bigl(f_n(X)\ne f_n(R_iX)\bigr).$  Influences play a central role in the theory of Boolean functions,
notably in junta approximation~\cite{Friedgut1998} and sharp-threshold
theory~\cite{FriedgutKalai1996}.

Thus, we have $\E C_n=I_n(f_n)$ by~\cite[Proposition~1.19]{JonassonSteif2016}. At bias $1/2$, the
Kahn--Kalai--Linial theorem~\cite{KahnKalaiLinial1988} makes $\E C_n$ diverge for every non-degenerate transitive sequence.
Tightness, however, neither rules out a
diverging first moment nor gives a uniform bound on the total influence.
Our proof therefore uses a different quantity, an auxiliary mixed jump
sum. Let $h_{n,d}$ be the orthogonal projection of
$f_n-\E_{\pi_n}f_n$ onto Fourier degrees $1,\ldots,d$ in
$L^2(\pi_n)$, and set
\[
  Q_{n,d}
  =\sum_{0<s<1}
    \Delta f_n(X_s^{(n)})\,
    \Delta h_{n,d}(X_s^{(n)}),
\]
where $\Delta g(X_s^{(n)})
  =g(X_s^{(n)})-g(X_{s-}^{(n)}).$

For each fixed $d$, transitivity and an $L^2\to L^4$
hypercontractive estimate~\cite{Bonami1970,Beckner1975,ODonnell2014}
imply, under~\eqref{eq:bias-condition}, that the maximum jump of
$h_{n,d}(X_t^{(n)})$ on $(0,1)$ tends to zero in probability. Booleanity gives the
pathwise bound
\[
  |Q_{n,d}|
  \le C_n\max_{0<s<1}
       |\Delta h_{n,d}(X_s^{(n)})|,
\]
while reversibility and martingale compensation yield
\[
  \sup_n \E Q_{n,d}^{\,2}<\infty.
\]
Thus tameness gives $Q_{n,d}\to0$ in probability and hence
$\E Q_{n,d}\to0$. Moreover, $\E Q_{n,d}$ equals twice the Dirichlet
energy of $h_{n,d}$, so the Fourier mass in every fixed range of positive
degrees vanishes. Tameness also bounds the high-degree Fourier
tail~\cite{JonassonSteif2016}, forcing the variance to vanish.

Our argument also yields an extension for regular functions.  A Boolean function is \emph{regular} if its
coordinate influences are all equal; every transitive Boolean function
is regular. Forsstr\"om~\cite[Theorem~1.4]{Forsstrom2021} proved
non-tameness for non-degenerate increasing regular sequences when the
biases are bounded away from zero and one. The following result removes
the increasing assumption.

\begin{corollary}\label{cor:regular}
Suppose that $\epsilon\le p_n\le1-\epsilon$ for all sufficiently large
$n$, for some $\epsilon>0$, and that each $f_n$ is regular. If $(f_n)$
is tame, then
\[
  \Var_{\pi_n}(f_n)\longrightarrow0.
\]
\end{corollary}

Section~\ref{sec:fourier} introduces the Fourier preliminaries and the spectral
consequence of tameness. Section~\ref{sec:jumps} proves the jump estimates,
and Section~\ref{sec:proof} proves the main theorem and
Corollary~\ref{cor:regular}.

\section{Preliminaries}
\label{sec:fourier}

Write \(v_n=p_n(1-p_n)\). For $S\subseteq[n],$ we use the biased Fourier basis
\[
 \chi_{n,S}(x)=\prod_{i\in S}\frac{x_i-p_n}{\sqrt{v_n}},
\]
with $\chi_{n,\varnothing}=1$. The family \((\chi_{n,S})_{S\subseteq[n]}\) forms an orthonormal basis in \(L^2(\pi_n)\), so
\[
 f_n=\sum_{S\subseteq[n]}\widehat f_n(S)\chi_{n,S},
\]
where $\widehat f_n(S)=\E_{\pi_n}[f_n\chi_{n,S}].$ We refer to~\cite{ODonnell2014} for background on Fourier analysis on product spaces. Let $W_{n,k}=\sum_{|S|=k}\widehat f_n(S)^2,$ $k\ge1$ with \(W_{n,k}=0\) for \(k>n\). Parseval's identity gives
\begin{equation}\label{eq:variance-spectrum}
 \Var_{\pi_n}(f_n)=\sum_{k\ge1}W_{n,k}\le\frac14.
\end{equation}

For \(a\in\{0,1\}\), let \(x^{i\to a}\) denote the configuration obtained from \(x\) by setting its \(i\)-th coordinate to \(a\). The generator of \(X^{(n)}\) is
\[
(L_ng)(x)
=
\sum_{i=1}^n
\Bigl[(1-p_n)g(x^{i\to0})+p_ng(x^{i\to1})-g(x)\Bigr],
\]
and the associated semigroup is \(e^{tL_n}\).

The dynamics is reversible with respect to \(\pi_n\). The Fourier characters
diagonalize \(L_n\):
\begin{equation}\label{eq:eigenfunctions}
    L_n\chi_{n,S}=-|S|\chi_{n,S},
    \quad \text{and} \quad
    e^{tL_n}\chi_{n,S}=e^{-t|S|}\chi_{n,S}.
\end{equation}
The associated Dirichlet form is $\cE_n(g,h)=-\ip{g}{L_nh}_{\pi_n}.$


For each integer \(d\ge1\), define a test function that retains low-order Fourier information
\begin{equation}\label{eq:projection}
 h_{n,d}=\sum_{1\le|S|\le d}\widehat f_n(S)\chi_{n,S},
\end{equation}
while $J_{n,i}^{(d)}=\sum_{\substack{S\ni i\\1\le|S|\le d}}\widehat f_n(S)^2$ is the Fourier energy of $h_{n,d}$ in the \(i\)-th coordinate.

Orthogonality yields
\begin{align}
 \cE_n(f_n,h_{n,d})=\cE_n(h_{n,d},h_{n,d})
   =\sum_{k=1}^d kW_{n,k}\le d,\label{eq:projection-energy}
\end{align}
and
\begin{equation}
\norm{L_nh_{n,d}}_{2,\pi_n}^2=\sum_{k=1}^d k^2W_{n,k}\le d^2.\label{eq:projection-generator}
\end{equation}

Summing over \(i\), each squared Fourier coefficient \(\widehat f_n(S)^2\)
with \(1\le |S|\le d\) is counted exactly \(|S|\) times. Hence
\begin{equation}\label{eq:sum-J}
    \sum_{i=1}^n J_{n,i}^{(d)}
    =\sum_{k=1}^d kW_{n,k}
    \le d.
\end{equation}

\begin{lemma}\label{lem:spectral-tail}
If \((f_n)\) is tame, then
\begin{equation}\label{eq:spectral-tail}
    \lim_{d\to\infty}\limsup_{n\to\infty}
    \sum_{k>d}W_{n,k}=0.
\end{equation}
\end{lemma}

\begin{proof}
Fix \(0<t<1\). On any output path with at most \(K\) changes, the set of
times \(s\in[0,1-t]\) for which $Y_{s+t}^{(n)}\ne Y_s^{(n)}$ has Lebesgue measure at most \(Kt\), since each change can lie in
\((s,s+t]\) for a set of starting times of length at most \(t\).
By stationarity,
\[
    (1-t)\PP\bigl(Y_t^{(n)}\ne Y_0^{(n)}\bigr)
    \le Kt+(1-t)\PP(C_n>K).
\]
Hence
\[
    \lim_{t\downarrow0}\limsup_{n\to\infty}
    \PP\bigl(Y_t^{(n)}\ne Y_0^{(n)}\bigr)=0,
\]
by first letting \(t\downarrow0\) and then \(K\to\infty\), using the
tightness of \((C_n)\).

Since \(Y_t^{(n)}\in\{0,1\}\),
\[
    \PP\bigl(Y_t^{(n)}\ne Y_0^{(n)}\bigr)
    =
    \E\bigl(Y_t^{(n)}-Y_0^{(n)}\bigr)^2.
\]
By stationarity and \eqref{eq:eigenfunctions},
\[
    \PP\bigl(Y_t^{(n)}\ne Y_0^{(n)}\bigr)
    =
    2\sum_{k\ge1}(1-e^{-kt})W_{n,k}.
\]
Taking \(t=1/(d+1)\), we obtain
\[
    \sum_{k>d}W_{n,k}
    \le
    \frac{
        \PP\bigl(
            Y_{1/(d+1)}^{(n)}\ne Y_0^{(n)}
        \bigr)
    }{2(1-e^{-1})}.
\]
Taking \(\limsup_{n\to\infty}\) and then \(d\to\infty\) proves the lemma.
\end{proof}

Lemma~\ref{lem:spectral-tail} is the spectral form of the implication from tameness to noise stability proved in~\cite{JonassonSteif2016}. 

\section{Jump estimates}
\label{sec:jumps}

\subsection{A mixed jump sum}

We first work with a stationary continuous time Markov chain \(X\) on a finite state space \(\mathsf E\). Denote its off-diagonal rates by \(r(x,y)\), with \(r(x,x)=0\), its generator by \(L\), and its stationary law by \(\pi\). For a function \(h\), write \(\Delta h(X_s)=h(X_s)-h(X_{s-})\).

We will use the standard compensation identities for jump processes. If \(G(x,x)=0\), set
\[
 S_t^G=\sum_{0<s\le t}G(X_{s-},X_s).
\]
Then \(M_t^G=S_t^G-\int_0^t a_G(X_s)\,ds\) is a square-integrable martingale, where $a_G(x)=\sum_y r(x,y)G(x,y)$, and, in stationarity,
\begin{align}
 \E S_t^G&=t\sum_{x,y}\pi(x)r(x,y)G(x,y),\label{eq:compensation-first}\\
 \E(M_t^G)^2&=t\sum_{x,y}\pi(x)r(x,y)G(x,y)^2.\label{eq:compensation-second}
\end{align}

Suppose now that \(\pi(x)>0\) for every \(x\) and that the chain is reversible, so \(\pi(x)r(x,y)=\pi(y)r(y,x)\). Its Dirichlet form satisfies
\begin{equation}\label{eq:dirichlet-edges}
 \cE(g,h)=-\ip{g}{Lh}_\pi
 =\frac12\sum_{x,y}\pi(x)r(x,y)(g(y)-g(x))(h(y)-h(x)).
\end{equation}

\begin{lemma}\label{lem:mixed-jumps}
Let \(X\) be stationary and reversible, and let \(f,h:\mathsf E\to\R\) satisfy \(\norm{f}_\infty\le1\). For
\[
 Q(f,h)=\sum_{0<s<1}\Delta f(X_s)\Delta h(X_s),
\]
we have
\begin{equation}\label{eq:mixed-mean}
 \E Q(f,h)=2\cE(f,h)
\end{equation}
and
\begin{equation}\label{eq:mixed-L2}
 \bigl(\mathbb{E}[Q(f,h)^2]\bigr)^{1/2}
 \le 2\bigl(\sqrt{2\cE(h,h)}+\norm{Lh}_{2,\pi}\bigr).
\end{equation}
\end{lemma}

\begin{proof}
The expectation follows directly from \eqref{eq:compensation-first} and \eqref{eq:dirichlet-edges}. To bound the second moment, note that 
\[
\begin{aligned}
Q(f,h)
&=\sum_{0<s<1}
\Delta f(X_s)\Delta h(X_s)\\
&=\sum_{0<s<1}
\bigl(f(X_s)-f(X_{s-})\bigr)\Delta h(X_s)\\
&=\sum_{0<s<1}f(X_s)\Delta h(X_s)-\sum_{0<s<1}f(X_{s-})\Delta h(X_s)\\
&=:B-A.
\end{aligned}
\] 

Apply compensation to \(G(x,y)=f(x)(h(y)-h(x))\). Its compensator is \(fLh\), and \eqref{eq:compensation-second} gives
\[
 \E(M_1^G)^2
 \le\sum_{x,y}\pi(x)r(x,y)(h(y)-h(x))^2
 =2\cE(h,h).
\]
Also, by Cauchy--Schwarz and stationarity,
\[
 \E\left|\int_0^1 f(X_s)Lh(X_s)\,ds\right|^2
 \le\int_0^1\E|Lh(X_s)|^2\,ds=\norm{Lh}_{2,\pi}^2.
\]
Consequently,
\begin{equation}\label{eq:A-bound}
 \bigl(\mathbb{E}[A^2]\bigr)^{1/2}
 \le \sqrt{2\cE(h,h)}+\norm{Lh}_{2,\pi}.
\end{equation}

The stationary path law on \([0,1]\) is invariant under time reversal. Indeed, the density of a path with states \(x_0,\ldots,x_m\) and jump times \(0<t_1<\cdots<t_m<1\) is
\[
 \pi(x_0)\prod_{j=1}^m r(x_{j-1},x_j)
       \prod_{j=0}^m e^{-\lambda(x_j)(t_{j+1}-t_j)},
\]
where \(\lambda(x)=\sum_y r(x,y)\), \(t_0=0\) and \(t_{m+1}=1\). Detailed balance leaves this density unchanged when the state sequence and holding intervals are reversed. A forward jump \(x\to y\) contributes \(f(y)(h(x)-h(y))\) to \(A\) on the reversed path, the negative of its contribution to \(B\) on the forward path. Thus \(B\) has the same distribution as \(-A\),
\[
\|Q(f,h)\|_{L^2}
\le \|A\|_{L^2}+\|B\|_{L^2}
=2\|A\|_{L^2},
\]
and \eqref{eq:A-bound} yields \eqref{eq:mixed-L2}.

\end{proof}


\subsection{Small jumps of low-degree projections}

We need a fourth-moment estimate with explicit dependence on the bias. The following bound is a form of hypercontractivity on product spaces; see~\cite{ODonnell2014} for the general theory.

\begin{lemma}\label{lem:fourth-moment}
Let \(P\) be a Fourier polynomial of degree at most \(r\) on a product Bernoulli\((p)\) space, where \(r\ge0\) is an integer, and put \(v=p(1-p)\). Then
\begin{equation}\label{eq:fourth-moment}
 \norm{P}_4\le v^{-r/2}\norm{P}_2.
\end{equation}
\end{lemma}

\begin{proof}
For one coordinate, let \(\chi=(X-p)/\sqrt v\), where \(X\) has the Bernoulli\((p)\) distribution. Then
\[
 \E\chi=0,\quad \E\chi^2=1,\quad
 \E\chi^3=\frac{1-2p}{\sqrt v},\quad
 \E\chi^4=\frac{1-3v}{v}.
\]
Define \(T(a+b\chi)=a+\sqrt v\,b\chi\). For real \(a,b\),
\begin{align*}
 \E(a+\sqrt v\,b\chi)^4
 &=a^4+6va^2b^2+4v(1-2p)ab^3+v(1-3v)b^4\\
 &\le a^4+8va^2b^2+(3v-3v^2)b^4\\
 &\le(a^2+b^2)^2.
\end{align*}
Here we used \(4v|a||b|^3\le2va^2b^2+2vb^4\) and \(v\le1/4\). Thus \(T\) is an \(L^2\)-to-\(L^4\) contraction.

This contraction tensorizes. To see this, separate a coordinate \(x\) from the others \(y\), and let \(T_x,T_y\) be the corresponding operators. If the contraction holds for \(T_y\), the one-coordinate bound and Minkowski's inequality give
\begin{align*}
 \norm{T_xT_yF}_{L^4_{x,y}}
 &\le\left\|\bigl(\E_x|T_yF|^2\bigr)^{1/2}\right\|_{L^4_y}\\
 &\le\left(\E_x\norm{T_yF(x,\cdot)}_{L^4_y}^2\right)^{1/2}
 \le\norm{F}_{L^2_{x,y}}.
\end{align*}
Induction proves the contraction in every dimension. On the Fourier basis, the resulting operator multiplies \(\chi_S\) by \(v^{|S|/2}\). If \(P=\sum_{|S|\le r}a_S\chi_S\), apply the contraction to
\[
 R=\sum_{|S|\le r}v^{-|S|/2}a_S\chi_S.
\]
Since \(TR=P\) and \(\norm{R}_2\le v^{-r/2}\norm{P}_2\), we obtain \eqref{eq:fourth-moment}.
\end{proof}

Return to the product resampling chain, and define $M_{n,d}=\max_{0<s<1}|\Delta h_{n,d}(X_s^{(n)})|,$ with the maximum equal to zero when there are no configuration jumps.

\begin{lemma}\label{lem:small-jumps}
For every integer \(d\ge1\) and every \(\eta>0\),
\begin{equation}\label{eq:small-jumps}
 \PP(M_{n,d}>\eta)
 \le\frac{2}{\eta^4v_n^{2d-1}}
       \sum_{i=1}^n\bigl(J_{n,i}^{(d)}\bigr)^2.
\end{equation}
\end{lemma}

\begin{proof}
Define \(D_i h(x)=h(x^{i\to1})-h(x^{i\to0})\), which is independent of \(x_i\). For the projection in \eqref{eq:projection},
\begin{equation}\label{eq:Di-hnd}
 D_i h_{n,d}
 =\frac1{\sqrt{v_n}}
 \sum_{\substack{S\ni i\\1\le |S|\le d}}
 \widehat f_n(S)\chi_{n,S\setminus\{i\}}.
\end{equation}
Hence \(\E_{\pi_n}|D_i h_{n,d}|^2=J_{n,i}^{(d)}/v_n\). \eqref{eq:Di-hnd} has degree at most \(d-1\), so Lemma~\ref{lem:fourth-moment} gives
\begin{equation}\label{eq:derivative-fourth}
 \E_{\pi_n}|D_i h_{n,d}|^4
 \le v_n^{-2d}\bigl(J_{n,i}^{(d)}\bigr)^2.
\end{equation}
By \eqref{eq:compensation-first},
\begin{align*}
 \E\sum_{0<s<1}|\Delta h_{n,d}(X_s^{(n)})|^4
 &=\sum_x\pi_n(x)\sum_{i=1}^n
             [p_n(1-x_i)+(1-p_n)]x_i|D_i h_{n,d}(x)|^4\\
 &=2v_n\sum_{i=1}^n\E_{\pi_n}|D_i h_{n,d}|^4\\
 &\le2v_n^{-(2d-1)}\sum_{i=1}^n\bigl(J_{n,i}^{(d)}\bigr)^2.
\end{align*}
where $x^i=x^{i\to 1-x_i}$. For the second equality, condition on all coordinates other than \(i\): the derivative is fixed and the mean jump rate is \((1-p_n)p_n+p_n(1-p_n)=2v_n\). Since \(M_{n,d}^4\) is at most the sum of fourth powers of all jumps, Markov's inequality proves \eqref{eq:small-jumps}.
\end{proof}

\section{Proof of the main results}
\label{sec:proof}

We first state the following spectral criterion.

\begin{theorem}\label{thm:criterion}
Let \((f_n)\) be a sequence of Boolean functions. Suppose that, for every fixed integer \(d\ge1\),
\begin{equation}\label{eq:criterion}
 v_n^{-(2d-1)}\sum_{i=1}^n\bigl(J_{n,i}^{(d)}\bigr)^2
 \longrightarrow0.
\end{equation}
If \((f_n)\) is tame, then \(\Var_{\pi_n}(f_n)\to0\).
\end{theorem}

\begin{proof}
Assume tameness, fix \(d\ge1\), and recall $Q_{n,d}$. Lemma~\ref{lem:mixed-jumps} and \eqref{eq:projection-energy}--\eqref{eq:projection-generator} yield
\begin{equation}\label{eq:Q-bounds}
 \E Q_{n,d}=2\sum_{k=1}^d kW_{n,k},
\quad \text{and} \quad 
\bigl(\mathbb E Q_{n,d}^2\bigr)^{1/2}
\le 2(\sqrt{2d}+d).
\end{equation}
As \(f_n\) is Boolean, every nonzero output jump has absolute value one, and therefore
\begin{equation}\label{eq:pathwise-product}
 |Q_{n,d}|\le C_nM_{n,d}.
\end{equation}
By \eqref{eq:criterion} and Lemma~\ref{lem:small-jumps}, \(M_{n,d}\to0\) in probability. For any \(\delta>0\) and integer \(K\ge1\),
\[
 \PP(|Q_{n,d}|>\delta)
 \le\PP(C_n>K)+\PP(M_{n,d}>\delta/K).
\]
Taking \(\limsup_n\) and then letting \(K\to\infty\) shows that \(Q_{n,d}\to0\) in probability. The uniform \(L^2\) bound in \eqref{eq:Q-bounds} implies convergence in \(L^1\): for every \(\delta>0\),
\[
 \limsup_{n\to\infty}\E|Q_{n,d}|
 \le\delta+(2(\sqrt{2d}+d))\limsup_{n\to\infty}\PP(|Q_{n,d}|>\delta)^{1/2}
 =\delta.
\]
Letting \(\delta\downarrow0\) and using \eqref{eq:Q-bounds}, we obtain
\begin{equation}\label{eq:low-degree-vanishes}
 \sum_{k=1}^d W_{n,k}\le\sum_{k=1}^d kW_{n,k}\longrightarrow0
 \qquad\text{for every fixed }d.
\end{equation}
Finally, by Lemma~\ref{lem:spectral-tail} and \eqref{eq:variance-spectrum}, as $d\to\infty$
\[
 \limsup_{n\to\infty}\Var_{\pi_n}(f_n)
 \le\limsup_{n\to\infty}\sum_{k>d}W_{n,k}\longrightarrow0.
\]
\end{proof}

\begin{proof}[Proof of Theorem~\ref{thm:main}]
Fix \(d\ge1\). If \(f_n(\sigma x)=f_n(x)\), invariance of \(\pi_n\) under coordinate permutations gives \(\widehat f_n(\sigma(S))=\widehat f_n(S)\). The map \(S\mapsto\sigma(S)\) preserves cardinality and sends sets containing \(i\) to sets containing \(\sigma(i)\). Transitivity therefore makes \(J_{n,i}^{(d)}\) independent of \(i\). By \eqref{eq:sum-J},
\begin{equation}\label{eq:transitive-J}
 \sum_{i=1}^n\bigl(J_{n,i}^{(d)}\bigr)^2
 =\frac1n\left(\sum_{i=1}^n J_{n,i}^{(d)}\right)^2
 \le\frac{d^2}{n}.
\end{equation}
Since \(v_n=b_n(1-b_n)\ge b_n/2\),
\[
 v_n^{-(2d-1)}\sum_{i=1}^n\bigl(J_{n,i}^{(d)}\bigr)^2
 \le\frac{2^{2d-1}d^2}{nb_n^{2d-1}}\longrightarrow0
\]
by \eqref{eq:bias-condition}. Theorem~\ref{thm:criterion} now gives \(\Var_{\pi_n}(f_n)\to0\) under tameness. This contradicts non-degeneracy and proves the final assertion.
\end{proof}

\begin{proof}[Proof of Corollary~\ref{cor:regular}]
Conditioning on the coordinates other than \(i\) shows that
\begin{equation}\label{eq:influence-Fourier}
 I_{n,i}(f_n)=2v_n\E_{\pi_n}(D_i f_n)^2
            =2\sum_{S\ni i}\widehat f_n(S)^2.
\end{equation}
This is the resampling normalization used in~\cite{Forsstrom2020,Forsstrom2021}. Consequently,
\begin{equation}\label{eq:total-influence}
 I_n(f_n)=2\sum_{k\ge1}kW_{n,k}.
\end{equation}
Assume tameness. Since \(W_{n,k}=0\) for \(k>n\), for every fixed \(d\),
\[
 \frac{I_n(f_n)}n
 \le\frac{2d}{n}+2\sum_{k>d}W_{n,k}.
\]
Lemma~\ref{lem:spectral-tail}, with \(n\to\infty\) before \(d\to\infty\), yields \(I_n(f_n)/n\to0\). Regularity and \eqref{eq:influence-Fourier} give
\[
 \max_i J_{n,i}^{(d)}\le\frac{I_n(f_n)}{2n}.
\]
Together with \eqref{eq:sum-J}, 
\[
 \sum_{i=1}^n\bigl(J_{n,i}^{(d)}\bigr)^2
 \le\left(\max_i J_{n,i}^{(d)}\right)\sum_i J_{n,i}^{(d)}
 \le\frac{dI_n(f_n)}{2n}\longrightarrow0.
\]
The biases are bounded away from zero and one, so \(v_n\) is bounded below by a positive constant. Thus \eqref{eq:criterion} holds for every fixed \(d\), and Theorem~\ref{thm:criterion} proves the corollary.
\end{proof}

\subsection{A quantitative tail bound}

We can also provide a finite dimensional lower bound on the tail of \(C_n\) when a fixed range of Fourier degrees carries positive mass.

\begin{proposition}\label{prop:quantitative}
Let \(f_n\) be transitive, and suppose that \(\sum_{k=1}^d W_{n,k}\ge w>0\) for some integer \(d\ge1\). Then, for every integer \(K\ge1\),
\begin{equation}\label{eq:quantitative-tail}
 \PP(C_n>K)
 \ge\frac{w^2}{(2(\sqrt{2d}+d))^2}
     -\frac{2d^2K^4}{w^4nv_n^{2d-1}}.
\end{equation}
In particular, if the spectral mass bound holds eventually along a sequence, for fixed \(d,w\), and \(nv_n^{2d-1}\to\infty\), then for every fixed $K\ge1$,
\[
 \liminf_{n\to\infty}\PP(C_n>K)\ge\frac{w^2}{(2(\sqrt{2d}+d))^2}>0.
\]
\end{proposition}

\begin{proof}
Let \(Q_{n,d}\) be as in before. By \eqref{eq:Q-bounds}, \(\E Q_{n,d}\ge2w\) and \(\norm{Q_{n,d}}_2\le 2(\sqrt{2d}+d)\). Define the event
\[
\mathcal E_{n,d,K}
:=
\{C_n>K\}\cup\{M_{n,d}>w/K\}.
\]

On $\mathcal E_{n,d,K}^{\,c}$, \eqref{eq:pathwise-product} gives \(|Q_{n,d}|\le w\). Splitting the expectation over \(\mathcal E_{n,d,K}\) and \(\mathcal E_{n,d,K}^{\,c}\), and applying Cauchy--Schwarz, yields
\[
 2w\le\E|Q_{n,d}|
 \le w+(2(\sqrt{2d}+d))\bigl(\PP(C_n>K)+\PP(M_{n,d}>w/K)\bigr)^{1/2}.
\]
It follows that \(\PP(C_n>K)\ge w^2/(2(\sqrt{2d}+d))^2-\PP(M_{n,d}>w/K)\). Lemma~\ref{lem:small-jumps} and \eqref{eq:transitive-J} bound the last probability by \(2d^2K^4/(w^4nv_n^{2d-1})\), proving \eqref{eq:quantitative-tail}. The final assertion follows by taking \(\liminf_n\) for fixed \(K\).
\end{proof}

\section*{Acknowledgements}
This work was supported by the Shandong Provincial Natural Science
Foundation for Outstanding Youth Scholar (Grant No.~ZR2024JQ030).

\end{document}